\documentclass[12pt]{amsart}

\usepackage[sc,osf,slantedGreek]{mathpazo}
\usepackage[T1]{fontenc}
\usepackage{amssymb}
\usepackage{euscript}
\usepackage[all]{xy}
\usepackage{parskip}
\usepackage{enumitem}
\setlist[enumerate]{labelindent=-.5in,leftmargin=0pt}
\setlist[itemize]{labelindent=-.5in,leftmargin=0pt}
\usepackage[colorlinks=true]{hyperref}
\usepackage{color}
\definecolor{green}{RGB}{0,127,0}
\definecolor{red}{RGB}{191,0,0}

\newcommand{\B}[1]{{\mathbf #1}}

\makeatletter
\def\@settitle{\begin{center}%
  \baselineskip14\p@\relax
    {\Large\textit \@title}
  \end{center}%
}
\def\@setauthors{%
  \begingroup
  \def\thanks{\protect\thanks@warning}%
  \trivlist
  \centering\footnotesize \@topsep30\p@\relax
  \advance\@topsep by -\baselineskip
  \item\relax
  \author@andify\authors
  \def\\{\protect\linebreak}%
  {\scshape \authors}%
  \ifx\@empty\contribs
  \else
    ,\penalty-3 \space \@setcontribs
    \@closetoccontribs
  \fi
  \endtrivlist
  \endgroup
}
\makeatother

\newtheorem{theorem}[subsection]{Theorem}
\newtheorem{corollary}[subsection]{Corollary}
\newtheorem{lemma}[subsection]{Lemma}
\newtheorem{proposition}[subsection]{Proposition}
\theoremstyle{definition}
\newtheorem{definition}[subsection]{Definition}

\newtheorem{example}[subsection]{Example}
\newtheorem{examples}[subsection]{Examples}
\theoremstyle{remark}
\newtheorem{remark}[subsection]{Remark}

\numberwithin{figure}{section}
\numberwithin{table}{section}

\newcommand{\OP}{\operatorname}

\begin{document}

\title[Biinvariant metrics]{The cancelation norm and the geometry of biinvariant word metrics}
\author[Brandenbursky]{Michael Brandenbursky}
\address[M. Brandenbursky]{CRM, University of Montreal, Canada}
\email{michael.brandenbursky@mcgill.ca}
\author[Gal]{\'Swiatos\l aw R. Gal}
\address[S.R. Gal]{Uniwersytet Wroc\l awski}
\email{sgal@math.uni.wroc.pl}
\author[K\k{e}dra]{Jarek K\k{e}dra}
\address[J. K\k{e}dra]{University of Aberdeen and University of Szczecin}
\email{kedra@abdn.ac.uk}
\author[Marcinkowski]{Micha{\l} Marcinkowski}
\address[M. Marcinkowski]{Uniwersytet Wroc\l awski}
\email{marcinkow@math.uni.wroc.pl}
\thanks{M.M. was partially supported by a scholarship from the Polish Science Foundation. SG and MM were partially supported by Polish National Science Center (NCN) grant 2012/06/A/ST1/00259}

\begin{abstract}
We study biinvariant word metrics on groups. We provide
an efficient algorithm for computing the biinvariant word
norm on a finitely generated free group and we construct an isometric embedding
of a locally compact tree into the biinvariant Cayley
graph of a nonabelian free group. We investigate the geometry
of cyclic subgroups.
We observe that in many classes of groups cyclic subgroups
are either bounded or detected by homogeneous quasimorphisms.
We call this property the bq-dichotomy and we prove it for
many classes of groups of geometric origin.
\end{abstract}

\maketitle
\section{Introduction} \label{S:intro}

The main object of study in the present paper are biinvariant
word metrics on normally finitely generated groups. Let us recall definitions.
Let $G$ be a group generated by a symmetric set
$S\subset G$. Let $\overline{S}$ denote the
the smallest conjugation invariant subset of $G$
containing the set $S$. The word norm of an element
$g\in G$  associated with the sets $S$ and $\overline{S}$
is denoted by $|g|$ and $\|g\|$ respectively:
\begin{align*}
|g|&:=\min\{k\in \B N\,|\,g=s_1\cdots s_k, {\text{ where }} s_i\in S\},\\
\|g\|&:=
\min\{k\in \B N\,|\,g=s_1\cdots s_k, {\text{ where }} s_i\in \overline{S}\}.
\end{align*}

The latter norm is conjugation invariant and defined if $G$ is generated by
$\overline{S}$ but not necessarily by $S$. If $S$ is finite and $G$ is
generated by $\overline{S}$ then we say that $G$ is {\em normally finitely
generated}. This holds, for example, when $G$ is a simple group and
$S=\{g^{\pm 1}\}$ and $g\neq \OP{1}_G$.  Another example is the infinite braid
group $B_{\infty}$ which is normally generated by one element twisting the two
first strands.

\begin{remark}
The metric associated with the conjugation invariant norm
is defined by ${\bf d}(g,h):=\|gh^{-1}\|$. It is
biinvariant in the sense that both left and right actions
of $G$ on itself are by isometries. We focus in the paper
exclusively on both conjugation invariant word norms and
associated with them biinvariant metrics. Most of the arguments
and computations are done for norms.
\end{remark}

Since invariant sets are in general infinite,
the corresponding word norms are not considered by the classical
geometric group theory. The motivation for studying such norms
comes from geometry and topology because transformation groups
of manifolds often carry naturally defined conjugation
invariant norms. The examples include the Hofer norm and the autonomous
norm in symplectic geometry, fragmentation norms and the volume of the
support norm in differential geometry and others, see for example
\cite{MR2900661,BK,MR2390536,MR2866929,Hofer,Kotschick,LM}.

Biinvariant word metrics are at present not well understood.
It is known that for some nonuniform lattices in semisimple
Lie groups (e.g. $\OP{SL}(n,\B Z)$, $n\geq 3$)  biinvariant
metrics are bounded \cite{BIP,MR2819193}. In general, the
problem of understanding the
biinvariant geometry of lattices in higher rank semisimple
Lie groups is widely open.

The main tool for proving unboundedness of biinvariant word
metrics are homogeneous quasimorphisms. Thus if a group admits
a homogeneous quasimorphism that is bounded on a conjugation
invariant generating set then the group is automatically unbounded
with respect to the biinvariant word metric associated with
this set. Examples include hyperbolic groups and groups of
Hamiltonian diffeomorphisms of surfaces equipped with
autonomous or fragmentation metrics \cite{B,BK,EF}.
If a group $G$ is biinvariantly unbounded
it is interesting to understand what metric spaces can
be quasiisometrically embedded into $G$.

Before we discuss the content of the paper in greater detail
let us recall a basic property of biinvariant word metrics
on normally finitely generated groups.

\subsection*{Lipschitz properties of conjugation invariant norms
on nor\-mally fi\-ni\-te\-ly generated groups}\label{S:bi}

If a group $\Gamma$ is  normally finitely
generated then every homomorphism $\Psi\colon \Gamma\to G$ is Lipschitz
with respect to the norm $\|\circ\|$ on $\Gamma$ and any conjugation
invariant norm on $G$.
In particular, two choices of such a finite set $S$ produce
Lipschitz equivalent metrics, so in this case we will refer to
{\em the} word metric on a normally finitely generated group.
Also, such a metric is maximal among biinvariant metrics.

\subsection*{The cancelation norm}
Let $G$ be a group generated by a symmetric set $S$ and
let $w$ be a word in the alphabet $S$. The cancelation length
$|w|_{\times}$ is defined to be the least number of letters
to be deleted from $w$ in order to obtain a word trivial in $G$.
The cancelation norm of an element $g\in G$ is defined to be the
minimal cancelation length of a representing word.
We prove (Proposition \ref{P:cancellation})
that the cancelation norm is equal to the
conjugation invariant word norm associated with the generating set $S$.

In some cases the cancelation norm does not depend on
the representing word. In particular, the following result
is a consequence of a more general statement, see Proposition \ref{P:balanced}.

\begin{theorem}\label{T:raag_coxeter}
If\/ $G$ is either a right angled Artin group or a Coxeter group
then the cancelation norm of an element does not depend on the
representing word.
\end{theorem}

Section \ref{SS:algorithm} provides an efficient algorithm for
computing the cancelation length for nonabelian free groups.
More precisely, we prove the following result.

\begin{theorem} \label{T:algo}
Let $w \in \B F_n$ be a word of standard length $n$.
There exists an algorithm which computes the conjugation invariant word length
of $w$. Its complexity is $O(n^3)$ in time and $O(n^2)$ in memory.
\end{theorem}

A simple software for computing the biinvariant word norm on
the free group on two generators can be downloaded from the website of MM, see \cite{Mar}.

\subsection*{Quasiisometric embeddings}
One way of studying the geometry of a metric space $X$ is
to construct quasiisometric embeddings of understood
metric spaces into $X$. In Section \ref{SS:Zn}, we prove
that the free abelian group $\B Z^n$ with its standard
word metric can be quasiisometrically embedded into
a group $G$ equipped with the biinvariant word metric
provided $G$ admits at least $n$ linearly independent
homogeneous quasimorphisms.

We then proceed to embedding of trees.
We prove that there exists an isometric embedding
of a locally compact  tree in the biinvariant Cayley graph of a nonabelian
free group. We first construct an isometric embedding of the
one skeleton of the infinite unit cube
$$
\square^{\infty}:= \bigcup [0,1]^n
$$
equipped with the $\ell^1$-metric (Theorem \ref{T:cube}).
It is an easy observation that any locally compact tree
with edges of unit lengths admits an isometric embedding
into such a cube.

\begin{theorem}\label{T:tree}
Let $T$ be a locally compact tree with edges if unit lengths.
There is an isometric embedding $\OP{T}\to \B F_2$ into the Cayley graph of the
free group on two generators with the biinvariant word metric associated with
the standard generators.
\end{theorem}

\subsection*{The geometry of cyclic subgroups}
Let us recall that
a function $q\colon G \to \B R$
is called a {\em quasimorphism}
if there exists a real number $A\geq 0$ such that
$$
|q(gh) - q(g) - q(h)|\leq A
$$
for all $g,h\in G$. A quasimorphism $q$ is called {\em homogeneous}
if in addition
$$
q(g^n)=nq(g)
$$
for all $n\in \B Z$.
The vector space of homogeneous quasimorphisms on $G$
is denoted by $Q(G)$.
It is straighforward to prove that a quasimorphism
$q\colon G\to \B R$ defined on a normally finitely generated group
is Lipschitz with respect to the
biinvariant word metric on $G$ and the standard metric
on the reals~\cite[Lemma 3.6]{MR2819193}.
For more details about quasimorphisms and their
connections to different branches of mathematics see \cite{Calegari}.

The geometry of a cyclic subgroup $\langle g\rangle\subset G$ is
described by the growth rate of the function
$n\mapsto \|g^n\|$. A priori this function can be anything
from bounded to linear. If it is linear then the cyclic
subgroup is called {\em undistorted} and {\em distorted} otherwise.
It is an easy observation that if $\psi\colon G\to \B R$
is a homogeneous quasimorphism and $\psi(g)\neq 0$ then $g$
is undistorted. One of the main observations
of this paper is that for many classes of
groups of geometric origin
a cyclic subgroup is either bounded or detected
by a homogeneous quasimorphism.


\begin{definition}\label{D:bq}
A normally finitely generated group $G$ satisfies the
{\bf bq-dichotomy} if every cyclic subgroup of $G$ is
either {\bf b}ounded (with respect to the biinvariant word metric)
or detected by a homogeneous {\bf q}uasimorphism.
\end{definition}

\begin{remark}
One can consider
a weaker version of the above dichotomy when a cyclic
subgroup is either bounded or undistorted. Since undistortedness
is proved usually with the use of quasimorphism most of the
proofs yield the stronger statement. There is one exception
in this paper, Theorem \ref{T:coxeter}, where we prove
the weaker dichotomy for Coxeter groups
and the stronger under an additional
assumption. This is because we don't know how to extend
quasimorphisms from a parabolic subgroup of a Coxeter group.
More precisely, the following problem seems to be open:

{\em
Let $g\in W_T$, where $W_T$ is a standard parabolic subgroup
of a Coxeter group~$W$.
Does $\OP{scl}_{W_T}(g)>0$ imply  $\OP{scl}_{W}(g)>0$?
}

Here, $\OP{scl}_G$ denotes the stable commutator
length in $G$ (see Calegari's book \cite{Calegari} for details.)
\end{remark}

The only example known to the authors of a group which does not satisfy
bq-dichotomy is provided by Muranov in \cite{MR2718128}.  He constructs a group
$G$ with unbounded (but distorted) elements not detectable by a homogeneous quasimorphism.
His group $G$ is finitely generated but not finitely presented.  We know no
finitely presented example.  Also, we know no example of
an undistorted subgroup not detected by a homogeneous quasimorphism.

\begin{remark}
Observe that if $G$ satisfies the bq-dichotomy then
if $\OP{scl}_G(g)=0$ then the cyclic subgroup $\langle g\rangle$
is bounded, due to a theorem of Bavard \cite{Ba}.
\end{remark}

It is interesting to understand to what extent the bq-dichotomy
is true. To sum up let us make a list of groups that satisfy the bq-dichotomy:

\begin{itemize}
\item
Coxeter groups with even exponents - Theorem \ref{T:coxeter},
\item
finite index subgroups of mapping class groups of
closed oriented surfaces (possibly with punctures) - Theorem~\ref{T:mcg},
\item
Artin braid groups (both pure and full) on a finite number of strings - Theorem~\ref{T:braid},
\item
spherical braid groups (both pure and full) on a finite number of strings -
Theorem~\ref{T:sphere-braid},
\item
finitely generated nilpotent groups - Theorem \ref{T:nilpotent}. We actually prove
that the commutator subgroup $[G,G]$ is bounded in $G$,
\item
finitely generated solvable groups whose commutator subgroups
are finitely generated and nilpotent, e.g. lattices in simply connected
solvable Lie groups - Theorem \ref{T:solvable},
\item
$\OP{SL}(n,\B Z)$ - for $n=2$ it is proved by
Polterovich and Rudnick \cite{PR}; for $n>2$
the groups are bounded,
\item
lattices in certain Chevalley groups \cite{MR2819193} (the
groups are boun\-ded in this case),
\item
hyperbolic groups - due to  Calegari and Fujiwara
(Theorem 3.56 in \cite{Calegari}). They prove there that
if $g$ is a nontorsion element such that no positive power
of $g$ is conjugate to its inverse then it is detected by
a homogeneous quasimorphisms. On the other hand it follows
from Lemma \ref{L:bounded} that if a positive power of $g$ is
conjugate to its inverse then $g$ generates a bounded
cyclic subgroup.),
\item
right angled Artin groups - Theorem \ref{T:raag},
\item Baumslag-Solitar groups and fundamental groups of some
graph of groups - Theorem \ref{T:graph-of-gps}.
\end{itemize}

\subsection{Bounded elements}
\label{SS:bounded-elements}
Let $[x,y] = xyx^{-1}y^{-1}$ and ${}^t\!x = txt^{-1}$. In many cases we prove that an element $g\in G$ generates
a bounded cyclic subgroup by making the
observation that the element $[x,t]$ in the group
$$
\Gamma:= \left \langle x,t \,|\,\left [x,{}^t\!x\right ]=1\right \rangle
$$
generates a bounded subgroup of $\Gamma$. Then we construct nontrivial
homomorphism
$\Psi\colon\Gamma\to G$
such that
$\Psi[x,t]=g$. The examples include Baumslag-Solitar groups,
nonabelian braid groups $B_n$,
$\OP{SL}(2,\B Z[1/2])$, and HNN extensions of abelian groups, e.g.
$\OP{Sol}(3,\B Z)$, Hei\-senberg groups and lamplighter groups
(see Section \ref{SS:bounded}).

\subsection{Elements detected by a quasimorphism}
\label{SS:el-det-by-a-qm}

In some cases it is easy to provide examples of elements detected by a
nontrivial homogeneous quasimorphism, for example any nontrivial element in a
free group has this property.  Generalizing this observation yields the
following result (Section \ref{SS:quasi-res-real}).
\begin{theorem}\label{T:raag_cox}
Let $G$ be one of the following groups:
\begin{enumerate}
 \item a right angled Artin group,
 \item the commutator subgroup in a right angled Coxeter group,
 \item a pure braid group.
\end{enumerate}
Then for every nontrivial element
$g\in G$ there exists a homogeneous quasimorphism $\psi$ such that
$\psi(g)\neq 0$. In particular, every nontrivial cyclic subgroup in $G$
is biinvariantly undistorted.
\end{theorem}

We say that a group is {\em quasi-residually real} if
it satisfies the property from the statement of the above theorem.
Of course, a quasi-residually real group satisfies
the bq-dichotomy.

\section{The cancelation norm}
Let $G=\langle S\,|\,R\rangle$ be
a presentation of $G$, where $S$ is a finite symmetric
set of generators.
Let $w=s_1\dots s_n$ be a word in the alphabet $S$.
The number
$$
|w|_{\times}:=
\min\{k\in \B N\,|\, s_1\dots \widehat{s_{i_1}}\dots
\widehat{s_{i_k}}\dots s_n=1\text{ in }G\}
$$
is called the {\em cancelation length} of the word  $w$.
In other words, the cancelation length is the smallest number
of letters we need to cross out from $w$ in order to
obtain a word representing the neutral element.
The number
$$
|g|_{\times}:=
\min\{|w|_{\times }\in \B N\,|\, w \text{ represents $g$ in }G\}
$$
is called the {\em cancelation norm} of $g\in G$.

The sequence of indices $i_1,\ldots,i_k$ so that
deleting the letters $s_{i_1},\ldots,s_{i_k}$ makes
the word $w=s_1\dots s_n$ trivial is called the
trivializing sequence of $w$.
We will sometimes abuse the terminology and we
will call the sequence of letters $s_{i_1},\ldots,s_{i_n}$
trivializing. In this terminology the
cancelation length is the minimal length of
a trivializing sequence.

\begin{proposition}\label{P:cancellation}
Let\/ $G$ be finitely normally generated by a symmetric set $S\subset G$.
The cancelation norm is equal to the biinvariant word norm
associated with $S$.
\end{proposition}

\begin{proof}
Let $g=\prod_{i=1}^k w_i^{-1}s_iw_i$ then $(s_1,\dots,s_k)$ is a trivializing sequence for $g$,
and hence $|g|_\times \leq \|g\|$.

Let $g=u_0s_1u_1\cdots s_ku_k$ with $(s_1,\dots,s_k)$ being a trivializing sequence.
Then $g=\prod_{i=1}^k w_i^{-1}s_iw_i$ with $w_i=\prod_{j=i}^k u_j$.  Thus $\|g\| \leq |g|_\times$.
\end{proof}

Let $G=\langle S\,|\, R\rangle $. A relation $v=w$ in $R$ is called {\em balanced}
if it has the following property:
if $\bar v$ is the word obtained from $v$ by deleting $k$ letters
then there exist $k$ letters
in $w$ such that deleting them produces a word $\bar w$ such that
$\bar v=_G \bar w$ in $G$. The following lemma
is straightforward to prove and is left to the reader.

\begin{lemma}\label{L:balanced}
If\/ $G=\langle S\,|\, R\rangle$ and $v=w$
is a balanced relation in $R$ then
$$|xvy|_{\times}=|xwy|_{\times}$$
for any words $x,y$.
\end{lemma}

\begin{example}\label{E:cox-raag}
Coxeter groups and right angled Artin groups admit presentations whose all relations are balanced.
Indeed, observe that there exists a
presentation of a Coxeter group with
relations of the form $s=s^{-1}$ and $st\dots s=ts\dots t$
or $(st)^n=(ts)^n$. The presentation with balanced relations of
a right angled Artin group has relations of the form $st=ts$.
\end{example}

The proof of the following observation is straightforward and is left to the reader.
\begin{proposition}
\hfill
\begin{enumerate}
\item
Let\/ $G_i=\langle S_i|R_i\rangle$, for  $i\in\{1,2\}$, be two presentations whose all relations are balanced and
with disjoint $S_1$ and $S_2$.  Let\/ $R_0=\{s_1s_2=s_2s_1|s_i\in S_i\}$.
Then $\langle S_1\cup S_2|R_0\cup R_1\cup R_2\rangle$ is a presentation
of\/ $G_1\times G_2$ with all relations balanced.
\item
Let\/ $G_i=\langle S_i|R_i\rangle$, for  $i\in \{1,2\}$, be two presentations whose all relations are balanced.
Assume that the subgroups of\/ $G_1$ and\/ $G_2$ generated by
$T=S_1\cap S_2$ are isomorphic (by the isomorphism which is the identity on $T$).
Then $G_1*_{\langle T\rangle}G_2=\langle S_1\cup S_2|R_1\cup R_2\rangle$
has all relations balanced.
\end{enumerate}
\end{proposition}

\begin{proposition}\label{P:balanced}
Let $G=\langle S|R\rangle$ be a presentation whose all relations are balanced.  Let $u$ and $v$
be two words in alphabet $S$ representing the same element
$g\in G$.  Then $|v|_\times=|w|_\times$.
In particular, the cancelation norm of $g$ is equal to the
cancelation length of any word representing $g$.
\end{proposition}

\begin{remark}
The following proof is incomplete. We present a proof of a special
case in the erratum in Section \ref{S:erratum} at the end of the paper.
\end{remark}

\begin{proof}
Suppose that $v=xy$ and $w=xr^{-1}ty$, where $r=t$ is
a relation from $R$ and $x,y$ are any words. Then we have that
\begin{align*}
|w|_{\times}&=|xr^{-1}ty|_{\times}\\
&=|xr^{-1}ry|_{\times}\\
&=|xy|_{\times}=|v|_{\times},
\end{align*}
where the second equality follows from Lemma \ref{L:balanced}.
If the words $v$ and $w$ represent the same element in $G$
then $w$ can be obtained from $v$ be performing a sequence of
the operations above. This implies the statement.
\end{proof}

\begin{example}\label{E:baumslag-solitar}
Let $G=\langle x,t\,|\, x^5=tx^2t^{-1}\rangle$ be a Baumslag-Solitar
group. In this case the cancelation length is not well defined
since, for example, the cancelation lengths of $x^5$ and of
$tx^2t^{-1}$ are distinct but these words represent the
same element.
\end{example}

\begin{corollary}\label{C:isometric}
Let\/ $G$ be either a Coxeter group or a right angled Artin group
generated by a set\/ $S$. The inclusion $P_T\subset G$ of the
standard parabolic subgroup associated with a subset\/ $T\subset S$
is an isometry with respect to biinvariant word metrics associated
with the sets $T$ and $S$.\qed
\end{corollary}

\subsection{An algorithm for computing the cancelation norm on a free group}
\label{SS:algorithm}

\begin{lemma}\label{L:xw}
If $x$ is a generator of a free group $\B F_n$ and
$w\in \B F_n$ then
$$
\|xw\|=
\min\left\{1+\|w\|,\min\{\|u\|+\|v\|,\text{ where } w=ux^{-1}v\}\right\}.
$$
\end{lemma}
\begin{proof}
The sequence $x,x_1,\ldots,x_n$ is minimal trivializing for the
word $xw$ if and only if the sequence $x_1,\ldots,x_n$ is minimal
trivializing for the word $w$. This implies that if $x$ is contained
in a minimal trivializing sequence then $\|xw\|=1+\|w\|$.

Suppose that $x$ is not contained in a minimal sequence trivializing $xw$.
Then the word $w$ must contain a letter equal to $x^{-1}$ that is not
contained in a minimal trivializing sequence $x_1,\ldots, x_n$ for $w$
and with which $x$ may be canceled out.
This implies that $w = ux^{-1}v$ and there exists $k$ such that
the sequence $x_1,\ldots,x_k$ minimally trivializes $u$ and
$x_{k+1},\ldots,x_n$ minimally trivializes $v$.
This implies that
$$\|w\|=\|u\|+\|v\|.$$
\end{proof}


\begin{proof}[Proof of Theorem \ref{T:algo}]
Assume that we have a reduced word $v$ of standard length $k$
and we know biinvariant lengths of all its proper connected subwords.
We can compute $\|v\|$
in time $k$ by processing the word from the beginning to the end in order to find
patterns as in Lemma \ref{L:xw} and computing the minimum.

Let $w = w_1w_2 \ldots w_n$ be a reduced word written in the standard generators.
In order to compute $\|w\|$ we need to compute
biinvariant lengths of all its connected subwords $w_iw_{i+1}\ldots w_j$.
Thus we proceed as follows: first we compute biinvariant lengths of all words of standard length 3
(words of length 1 and 2 always have biinvariant lengths 1 and 2, respectively),
then biinvariant lengths of all words of standard length 4 and so on.

In order to find computational complexity of this problem
assume that we have computed biinvariant lengths of all connected subwords
of standard length less then $k$. There are no more then $n$ subwords of standard length $k$. Thus
to compute biinvariant length of all subwords of standard length $k$
we perform no more then $Cnk$ operations for some constant $C$.

Thus the complexity of our algorithm is
\begin{equation*}
\Sigma_{k=1}^n Cnk = O(n^3)
\end{equation*}
During computations we need to remember only
lengths of subwords. Since there are $O(n^2)$ subwords, we used $O(n^2)$ memory.
\end{proof}

\begin{remark}
There is no obvious algorithm computing the conjugation invariant norm even for groups
where the word problem is solvable. However, it follows from Proposition
\ref{P:balanced} that we can find an algorithm
for computing the conjugation invariant word norm for groups admitting a
presentation whose all relations are balanced and with solvable word problem. But even then, we need to check
all possible subsequences of the chosen word which makes the algorithm
exponential in time.
\end{remark}

\section{Quasiisometric embeddings}

\subsection{Quasiisometric embeddings of $\B Z^n$}\label{SS:Zn}
We say that a map
$$f\colon(X,d_X)\to (Y,d_Y)$$
is a quasiisometric embedding if $f$ is a quasiisometry on its image.

\begin{lemma}[\cite{BK}]\label{L:qm}
Suppose that $\dim Q(G)\geq n$. Then there
exist $n$ quasimorphisms $q_1,\ldots,q_n\in Q(G)$ and
$g_1,\ldots,g_n\in G$ such that $q_i(g_j)=\delta_{ij}$.
\end{lemma}

\begin{theorem}\label{T:Zn->F2}
Suppose that $\dim Q(G)\geq n$.
Then there exists
a quasiisometric embedding $\B Z^n\to G$, where
$\B Z^n$ is equipped with the standard word metric and $G$ is equipped with the biinvariant word metric.
\end{theorem}
\begin{proof}
Let $q_1,\ldots,q_n\colon G\to \B R$ be
linearly independent homogeneous quasimorphisms
and let $g_1,\ldots,g_n\in G$ be such that
$q_i(g_j)=\delta_{ij}$, where $\delta_{ij}$ is
the Kronecker delta.

We define $\Psi\colon \B Z^n\to G$ by
$\Psi(k_1,\ldots,k_n)=g_1^{k_1}\cdots g_n^{k_n}$
and observe that
$$
\left\|\prod_ig_i^{k_i}\right\|\leq c\sum_i|k_i|,
$$
where $c=\max_i\|g_i\|$.
On the other hand, for every $j\in \{1,...,n\}$
we have
$$
c_j\left\|\prod_ig_i^{k_i}\right\|
\geq \left|q_j\left(\prod_ig_i^{k_i}\right)\right|
\geq |k_j|-nd_j,
$$
where $d_j$ is the defect of the quasimorphism $q_j$
and $c_j$ is its Lipschitz constant. Taking
$C:=\max\{c,nc_1,...,nc_n\}$ and $D:=C\sum_i nd_i$
and combining the two
inequalities we obtain
$$
\frac{1}{C}\sum_i|k_i|-D \leq \left\|\prod_ig_i^{k_i}\right\|
\leq C\sum_i|k_i|.
$$
\end{proof}

It follows from the above theorem that if the
space of homogeneous quasimorphisms of a group
$G$ is infinite dimensional then there exists
a quasiisometric embedding $\B Z^n\to G$ for
every natural number $n\in \B N$.

\begin{examples}\label{E:Zn->G}
Groups for which the space of homogeneous quasimorphisms
is infinite dimensional include:
\begin{enumerate}
\item
a nonabelian free group $\B F_m$ \cite{Brooks},
\item
Artin braid groups on 3 and more strings, and braid groups of a hyperbolic surface \cite{MR1914565},
\item
a non-elementary hyperbolic group \cite{EF},
\item
a finitely generated group which satisfies the small cancelation condition $C'(1/12)$ \cite{AD},
\item
mapping class group of a surface of positive genus  \cite{MR1914565},
\item
a nonabelian right angled Artin group \cite{MR2874959},
\item
groups of Hamiltonian diffeomorphisms of compact orientable surfaces \cite{EP,GG}.
\end{enumerate}
\end{examples}

\subsection{Embeddings of trees}\label{SS:tree}

\begin{theorem}\label{T:cube}
There is an isometric embedding $\square^{\infty}\to \B F_2$
of the vertex set of the infinite dimensional unit cube
with the $\ell^1$-metric
into the
free group on two generators with the biinvariant word metric coming from the standard generators.
\end{theorem}

\begin{proof}
Let $\B F_2$ be the free group generated by elements $a$ and $b$ and let
$\square^n = \{0,1\}^n$ denote the n-dimensional cube.
Let $\square^n$ be embedded into $\square^{n+1}$ as $\square^n \times \{0\}$.
For an arbitrary isometric embedding
$$\psi_n \colon \square^n \to~\B F_2$$
we construct an extension to
$$\psi_{n+1}\colon~\square^{n+1}~\to~\B F_2$$
as follows. Take an element $g = b^{4k}ab^{-4k}$, where
$k > |\psi(v)|$ for every
$v \in \square^n$. Define $\psi_{n+1}(v,0) = \psi_n(v)$ and $\psi_{n+1}(v,1) = g\psi_n(v)$.
Since the multiplication from the left is an isometry of the biinvariant metric,
$\psi_{n+1}$ is an isometry on both $\square^n \times \{0\}$ and $\square^n \times \{1\}$.
Hence what we need to show is that
$$
d((v,0) , (w,1)) =
\|\psi_{n+1}(v,0)\psi_{n+1}(w,1)^{-1} \|
$$
for every $v$, $w \in \square^n$. From the definition of $\psi_{n+1}$
we have that
$$
\|\psi_{n+1}(v,0)\psi_{n+1}(w,1)^{-1} \| =
\|\psi_n(v)\psi_n(w)^{-1}b^{4k}a^{-1}b^{-4k}\|
$$
We shall show that every minimal sequence trivializing
$$\psi_n(v)\psi_n(w)^{-1}b^{4k}a^{-1}b^{-4k}$$
contains the last letter $a^{-1}$, thus has the length
$$\|\psi_n(v)\psi_n(w)^{-1}\| + 1 = d((v,0) , (w,1)).$$
To see that assume on the contrary,
that $a^{-1}$ is not in a minimal
trivializing sequence. Then it has to cancel out with
some letter $a$ in $\psi_n(v)\psi_n(w)^{-1}$. But
$|\psi_n(v)\psi_n(w)^{-1}| < 2k$, so in order to
make the cancelation possible, one has to cross out
at least $2k+1$ letters $b$ between $\psi_n(v)\psi_n(w)^{-1}$ and $a^{-1}$. Since
$$2k+1 > |\psi_n(v)\psi_n(w)^{-1}|+1 \geq \|\psi_n(v)\psi_n(w)^{-1}\|+1,$$
such trivializing sequence cannot be minimal.

Now take an arbitrary $\psi_0$ and construct a sequence of
isometries $\psi_n$. Then
$
\psi_{\infty} = \bigcup_{n=0}^{\infty} \psi_n
$
is an isometric embedding of $\square^{\infty}$.
\end{proof}

\begin{proof}[Proof of Theorem \ref{T:tree}]
Let $T$ be a locally compact tree with edges of unit length.
Then $T$ isometrically embeds into the cube $\square^{\infty}$ as follows.
Let $v$ be a vertex of $T$ and $w$ be a vertex of $\square^{\infty}$.
We map a star of $v$ isometrically into a star of $w$. We then continue
the procedure inductively. It is possible because the star of
any vertex of the cube has countably infinitely many edges.
\end{proof}

\section{Biinvariant geometry of cyclic subgroups}\label{S:cyclic}

\subsection{Bounded cyclic subgroups}\label{SS:bounded}

\begin{lemma}\label{L:bounded}
Let $\Gamma:=\left\langle x,t\,\big |\left[x,{}^t\!x\right] =1\right\rangle$.
The following identity holds in~$\Gamma$:
$$
[x,t]^n=\left[x^n,t\right].
$$
In particular, the cyclic subgroup generated by
$[x,t]$ is bounded by two (with respect to the
generating set $\left\{x^{\pm 1},t^{\pm 1}\right\}$).
\end{lemma}
\begin{proof}
The identity is true for $n=1$. Let us assume that it is
true for some $n$. We then obtain that
\begin{align*}
[x,t]^{n+1}&=x^ntx^{-n} \left(t^{-1}xt\right) x^{-1}t^{-1}\\
&= x^n t \left(t^{-1}xt\right)x^{-n}x^{-1}t^{-1}\\
&= x^{n+1}tx^{-(n+1)}t^{-1}.	
\end{align*}
The statement follows by induction.
\end{proof}

\begin{examples}\label{E:bounded}
In the following examples we prove boundedness of a cyclic
subgroup of a group $G$ by constructing a relevant homomorphism
$\Psi\colon \Gamma\to G$.
\begin{enumerate}
\item
Let
$$\OP{BS}(p,q)=\langle a,t\,|ta^pt^{-1}=a^q\rangle$$
be the Baumslag-Solitar group, where $q>p$ are integers.
Let $\Psi\colon \Gamma\to \OP{BS}(p,q)$
be defined by $\Psi(x)=a^p$ and $\Psi(t)=t$. It follows
that the cyclic subgroup generated by
$[\Psi(t),\Psi(x)]$ is bounded.  Since $[t,a^p]=a^{p-q}$
we obtain that the cyclic subgroup generated by $a$ is bounded.
\item
Let $A\in \OP{SL}(2,\B Z)$ and let $G=\B Z\ltimes _A\B Z^2$
be the associated semidirect product. If
$A=\left(
\begin{smallmatrix}
a & b\\ c & d
\end{smallmatrix}
\right )$
then $G$ has the following presentation
$$
G =
\left\langle x,y,t\,\big |\,[x,y]=1, {}^t\!x=x^ay^c,{}^t\!y=x^by^d\right\rangle.
$$
Note that $\Psi\colon \Gamma \to G$ given by
$\Psi(t)=t$ and $\Psi(x)\in \B Z^2\subset G$ is a well defined
homomorphism.

If $A$ has two distinct real eigenvalues, for example if
$A$ is the Arnold cat matrix, then every element in the
kernel generates a bounded cyclic subgroup.
If $A\neq \OP{Id}$ has eigenvalues
equal to one then the center of $G$ is bounded
(cf. Theorem \ref{T:nilpotent} and \ref{T:solvable}).

\item
Consider the integer lamplighter group
$$
\B Z\wr \B Z = \B Z\ltimes \B Z^{\infty}.
$$
where $\B Z^{\infty}$ denotes the group of all
integer valued sequences $\{a_i\}_{i\in \B Z}$.
The generator $t$ of $\B Z$ acts by the shift and
hence the conjugation of $\{a_i\}$ by $t$ has the following
form
$$
t\{a_i\}t^{-1}=\{a_{i+1}\}.
$$
Since for every sequence $\{a_i\}$ there exists a
sequence $\{b_i\}$ such that $a_i=b_{i+1}-b_i$ we get that
$\{a_i\}=t\{b_i\}t^{-1}\{b_i\}^{-1}$. Let
$\Psi\colon \Gamma\to \B Z\wr \B Z$
be defined by $\Psi(x)=\{b_i\}$ and $\Psi(t)=t$.
This shows that every element in the commutator
subgroup of the lamplighter group generates a bounded
cyclic subgroup.
%
\item
Let $G = \OP{SL}(2,\B Z[1/2])$. Define
\begin{align*}
\Psi(x) &= \left(
\begin{smallmatrix}
1 & 1\\
0 & 1
\end{smallmatrix}
\right )\\
\Psi(t) &=
\left(
\begin{smallmatrix}
2 & 0\\
0 & 2^{-1}
\end{smallmatrix}
\right )\\
\end{align*}
It well defines a homomorphism since
$\Psi({}^t\!x) = \left (
\begin{smallmatrix}
1 & 4\\
0 & 1
\end{smallmatrix}
\right ).
$
Consequently we get that
$\left(
\begin{smallmatrix}
1 & -3\\
0 & 1
\end{smallmatrix}
\right )
=\Psi([x,t])$ generates a bounded cyclic subgroup.
More generally, it implies that the subgroups of
elementary matrices are bounded.
It is known that every element of $G$ can be written
as a product of up to five elementary matrices \cite{MR744962}
(see also \cite[Example 5.38]{Calegari}).
Hence we obtain that the whole group $G$ is
bounded.

\item \label{I:twist}
Let $B_k$ be the braid group on $k\geq 2$ strings and let
$i \colon B_n \to B_{2n}$ be a natural inclusion
on the first $n$ strings. Assume, that $g$ is in the
image of $i$. Let
$\Delta = (\sigma_1\ldots\sigma_{n-1})\ldots(\sigma_1\sigma_2)(\sigma_1)$
($\Delta$ is a half-twist Garside fundamental braid)
where $\sigma_i$'s are the standard Artin generators of the braid group $B_n$.
The conjugation
$\Delta g \Delta^{-1}$ flips $g$, thus
$[\Delta g \Delta^{-1},g] = e$.
For example, if $g = \sigma_1 \in B_4$, then
$\Delta g \Delta^{-1} = \sigma_3$ and
$\sigma_1\sigma_3^{-1}$ is
bounded in $B_4$.

\item \label{I:B_3}
Let $\Delta \in B_n$ be as above and let
$g=\sigma_{i_1}\dots\sigma_{i_k}\in B_n$ be any element.
The conjugation by $\Delta$ acts on $g$ as follows
$$
\Delta \sigma_{i_1}\sigma_{i_2}\dots \sigma_{i_k}\Delta^{-1}
=\sigma_{n-i_1}\sigma_{n-i_2}\dots \sigma_{n-i_k}.
$$
This implies that every braid of the form
$$
g=\sigma_{i_1}\sigma_{i_2} \dots \sigma_{n-i_2}^{-1}\sigma_{n-i_1}^{-1}
$$
is conjugate via $\Delta$ to its inverse.
Consequently, $[g^n,\Delta]=g^{2n}$ which implies that
the cyclic subgroup generated by $g$ is bounded by $2\|\Delta\|+\|g\|$.
For example, $\sigma_1\sigma_2^{-1}\in B_3$ generates a bounded
cyclic subgroup.
\item
It is a well-known fact that the center of $B_3$ is a cyclic
group generated by $\Delta^2$
(for definition of $\Delta$ see item (\ref{I:twist}) above).
We have a central extension
$$
1 \rightarrow \langle \Delta^2 \rangle
\rightarrow B_3 \xrightarrow{\Psi} \OP{PSL}(2,\B{Z})
\rightarrow 1
$$
where $\Psi(\sigma_1) = \left (
\begin{smallmatrix}
1 & 1\\
0 & 1
\end{smallmatrix} \right )$
and
$\Psi(\sigma_2) = \left (
\begin{smallmatrix}
1 & 0\\
-1 & 1
\end{smallmatrix} \right )$.
Denote
$$J = \Psi(\Delta) = \left (
\begin{smallmatrix}
0 & 1\\
-1 & 0
\end{smallmatrix} \right ).$$
Let $M \in \OP{PSL}(2,\B{Z})$ be a symmetric matrix.
It has two orthogonal eigenspaces (over $\B{R}$) with reciprocal eigenvalues.
The rotation $J$ swaps the eigenspaces which implies $M^J = M^{-1}$.
Moreover, there exists a braid $g$ in $B_3$ such that $g$ is conjugate
to $g^{-1}$ and $\Psi(g)=M$.  Indeed, any symmetric matrix is of
the form $[J,N]$ for some $N\in\OP{PSL}(2,\B Z)$.
Let $h$ be a lift of $N$ to $B_3$ and take $g=[\Delta,h]$.
Then
$$
\Delta^{-1}g\Delta=
h\Delta^{-1}h^{-1}\Delta=
h\Delta^{-1}\Delta^2h^{-1}\Delta^{-2}\Delta=
[h,\Delta]=g^{-1}.
$$
By the same argument as in item (\ref{I:B_3}) above, $g$ generates a bounded subgroup.
For example the image of an element $\sigma_1\sigma_2^{-1}$ is Arnold's cat matrix
$\left ( \begin{smallmatrix}
2 & 1\\
1 & 1
\end{smallmatrix} \right )$.
Since there are infinitely many conjugacy classes of symmetric matrices in
$\OP{PSL}(2,\B{Z})$, there are infinitely many conjugacy classes of bounded
cyclic subgroups in $B_3$. It should be compared to the group of pure braids
$P_3$, which is a finite index subgroup of $B_3$, but due to Theorem \ref{T:pure}
every nontrivial element in $P_3$ is undistorted.
\item
Let $f,h\colon M\to M$ be homeomorphisms of a manifold such that
$h(\OP{supp}(f))\cap\OP{supp}(f)=\emptyset$. Then the commutator
$[f,h]$ is bounded with respect to any biinvariant metric on
a group of homeomorphisms containing $f$ and $h$.
\end{enumerate}
\end{examples}

\subsection{Unbounded cyclic subgroups not detected by quasimor\-phisms}

Let $G$ be the simple finitely generated group constructed by
Muranov in \cite{MR2718128}. The following facts are proved in the Main Theorem
of his paper:
\begin{itemize}
\item
every cyclic subgroup of $G$ is distorted with respect to the
biinvariant word metric; in particular,
$G$ does not admit nontrivial homogeneous quasimorphisms
(Main Theorem (3)).
\item
$G$ contains cyclic subgroups unbounded with respect to the
commutator length (Main Theorem (1)); in particular, they
are unbounded with respect to the biinvariant word metric.
\end{itemize}

\subsection{Cyclic groups detected by homogeneous quasimorphisms}
\label{SS:quasi-res-real}

A group $G$ is called {\em quasiresidually real} if
for every element $g\in G$ there exists a homogeneous
quasimorphism $q\colon G\to \B R$ such that $q(g)\neq~0$.
It is equivalent to the existence of an unbounded quasimorphism
on the cyclic subgroup generated by $g$.

Free groups are quasiresidually real as well as torsion
free hyperbolic groups. It immediately follows that every
element in such a group is undistorted.
The purpose of this section is to prove the following results.

\begin{theorem}\label{T:raag}
A right angled Artin group is quasiresidually real.
\end{theorem}

\begin{theorem}\label{T:racg}
A commutator subgroup of a right angled Coxeter group is quasiresidually real.
\end{theorem}

\begin{theorem}\label{T:pure}
A pure braid group on any number of strings is quasiresidually real.
\end{theorem}

We need to introduce some terminology and state some lemmas
before the proof. The definitions and basic properties of rank-one elements
can be found in \cite{MR2507218}.

\begin{lemma}[Bestvina-Fujiwara]\label{L:b-f}
Assume that\/ $G$ acts on a proper $\OP{CAT}(0)$ or hyperbolic space $X$ by isometries
and $g \in G$ is a rank-one isometry.
If no positive power of\/ $g$ is conjugate to a positive power\/ of $g^{-1}$
then there is a homogeneous quasimorphism $q\colon G\to \B R$ which is
nontrivial on the cyclic subgroup generated by $g$.
\end{lemma}

\begin{proof}
Let $x_0\in X$ be the basepoint and $\sigma = [x_0, gx_0]$ be a geodesic interval.
If $\alpha$ is a piecewise geodesic path in $X$ then
let $|\alpha|_g$ be the maximal number of nonoverlaping translates
of $\sigma$ in $\alpha$ such that every subpath of $\alpha$ which connects
two consecutive translates of $\sigma$ is a geodesic segment.
Let $c_g\colon G\times G\to \B R$
be defined by
$$
c_g(x,y) := \inf_{\alpha} (|\alpha| - |\alpha|_g),
$$
where $\alpha$ ranges over all piecewise geodesic paths from $x$ to $y$.

Let $\Psi_g\colon G\to \B R$ be defined by
$$\Psi_g(h)=c_g(x_0,h(x_0)) - c_g(h(x_0),x_0)$$
and it follows from \cite[the proof of Theorem 6.3]{MR2507218} that there exists
$k>~0$ such that $\Psi_{g^k}$ is unbounded on the
cyclic group generated by $g$. Homogenizing $\Psi_{g^k}$
yields a required quasimorphism $q\colon G\to~\B R$.
\end{proof}

\begin{lemma}\label{L:not-tilda}
Let $G$ be a group acting on a proper
$\OP{CAT}(0)$ space $X$ by isometries.
Assume that $g \in G$ is a rank one isometry.
Then
$$xg^nx^{-1}\neq g^{-m}$$
for all $x\in G$ and $m,n>0$ provided that $m \neq n$.
If $G$ is torsion free the above holds also if $m=n$.
\end{lemma}
\begin{proof}
Suppose otherwise that there exists
$x\in G$ and $m,n$ such that
\begin{equation}\label{Eq:conj}
xg^nx^{-1}= g^{-m}.
\end{equation}
Assume that $m=n$. Then we have that
$$
x^2 g^n x^{-2} = xg^{-n}x^{-1} = g^n,
$$
which means that $g^n$ and $x^2$ commute.
Moreover, a group generated by $g^n$ and $x^2$ is of rank two.
To prove it assume otherwise that there exist $r \in G$
and $k,l$ such that $g^n = r^l$ and $x^2 = r^k$.
Take the $k$-th power of \eqref{Eq:conj}
$$
x g^{kn} x^{-1} = g^{-kn}.
$$
Together with $g^{kn} = r^{kl} = x^{2l}$ it gives that $x^{4l} = e$.
Which is a contradiction.

Since $g^n$ is an element of the free abelian subgroup of rank two,
it follows from the flat torus theorem that its axis lies in some flat.
Thus $g^n$, and consequently $g$, cannot be a rank one isometry.

Assume now that $m\neq n$ and take the $k$-th power of \eqref{Eq:conj}
$$
x g^{kn} x^{-1} = g^{km}.
$$
Let $P$ be a point on the axis $L\subset X$
on which $g$ acts by a translation by $d$ units.
Since $xg^{km}x^{-1}(P) = (g^{-1})^{kn}(P)$, the image of
a geodesic between $x^{-1}(P)$ and $g^{km}x^{-1}(P)$
with respect to  $x$ is contained in the axis $L$.

Let $l:=\OP{d}(x^{-1}(P),P)$,
where $\OP{d}$ is the distance function on $X$.
Applying the triangle inequality we get that
\begin{align*}
kmd&= \OP{d}\left(P,g^{km}(P)\right)\\
&\leq
\OP{d}\left(P,x^{-1}(P)\right) +
\OP{d}\left(x^{-1}(P),g^{km}x^{-1}(P)\right) +
\OP{d}\left(g^{km}x^{-1}(P),g^{km}(P)\right)\\
&= 2l + \OP{d}\left(x^{-1}(P), g^{km}x^{-1}(P)\right).
\end{align*}
This and a similar additional computation imply that
$$
kmd - 2l\leq
\OP{d}\left(P, xg^{km}x^{-1}(P)\right)\leq kmd +2l.
$$
On the other hand,  $\OP{d}\left((g^{-1})^{kn}(P),P\right)=knd$
which implies that
$$
(g^{-1})^{kn}(P)\neq xg^{km}x^{-1}(P)
$$
for $k$ large enough which contradicts \eqref{Eq:conj}.
\end{proof}

Let $A_{\Delta}$ be the right angled Artin group defined
by the graph $\Delta$. The presentation complex $X_{\Delta}$
of $A_{\Delta}$ is a two dimensional complex with one
vertex and with edges corresponding to generators and
two dimensional cells corresponding to relations. It is
a union of two dimensional tori. Its universal covering
$\widetilde{X}_{\Delta}$ is a $\OP{CAT}(0)$ square complex.
Let $\Delta'\subset \Delta$ be a full subgraph. Then
\begin{enumerate}
\item
the homomorphism $\pi\colon A_{\Delta}\to A_{\Delta'}$
defined by
$$
\pi(v):=
\begin{cases}
v & \text{ if } v\in \Delta'\\
1 & \text{ if } v\notin \Delta'
\end{cases}
$$
is well defined and surjective;
\item
every quasimorphism $q\colon A_{\Delta'}\to \B R$
extends to $A_{\Delta}$.
\end{enumerate}

If $\Delta'$ is a bipartite graph then the subgroup
$A_{\Delta'}\subset A_{\Delta}$ is called a {\em join}
subgroup.

\begin{proof}[Proof of Theorem \ref{T:raag}]
Let $g\in A_{\Delta}$ be a nontrivial element of a right angled Artin group.
Suppose that no conjugate of $g$ is contained in a join subgroup. Then,
according to Berhstock-Charney \cite[Theorem 5.2]{MR2874959}, $g$ acts on
the universal cover $\widetilde{X}_{\Delta}$ of the presentation complex
as a rank one isometry.

Thus, since $A_{\Delta}$ is torsion-free, we can apply Lemma \ref{L:not-tilda}
and consequently Lemma \ref{L:b-f} to $g$.

If $g$ is an element of a join subgroup then we project it to
one of the factors repeatedly until no conjugate of $g$
is contained in a join subgroup and then we apply the
above construction and extend the obtained quasimorphism
to $A_{\Delta}$.
\end{proof}

The right angled Coxeter group given by the graph $\Delta$ is a group
defined by the following presentation
$$
W_{\Delta} = \langle v \in \Delta |\thinspace v^2=1 ,
[v,v']=1 \text{ iff } (v,v') \text{ is an edge in } \Delta \rangle,
$$
As in the case of right angled Artin groups,
we have a well defined projection $\pi$ for an arbitrary full subgraph $\Delta'$
and the notion of a join subgroup.

The natural $\OP{CAT}(0)$ complex on which $W_{\Delta}$ acts geometrically is
the Davis cube complex $\Sigma_{\Delta}$ (see Davis \cite{MR2360474}
for more details).

\begin{proof}[Proof of Theorem \ref{T:racg}]
First we prove that the commutator subgroup $W_{\Delta}'$
of $W_{\Delta}$ is torsion-free. Let $g\in W_{\Delta}'$ be a torsion element.
By the $\OP{CAT}(0)$ property it stabilizes a cube in $\Sigma_{\Delta}$.
It follows from the definition of the Davis complex that
stabilizers of cubes are conjugate to spherical subgroups (i.e. subgroups
generated by vertices of some clique). Note that an abelianization of
$W_{\Delta}$ equals $\oplus_{v \in \Delta} \B{Z} / 2\B{Z}$
and spherical subgroups, as well as its conjugates project injectively
into the abelianization. Thus $g$ is a trivial element.

Now the argument is analogous to the proof of Theorem \ref{T:raag}.
Suppose that $g\in~W_{\Delta}'$ is an element such that no conjugate
of $g$ is contained in a join subgroup. According to \cite[Proposition 4.5]{MR2585575},
$g$ acts on $\Sigma_{\Delta}$ as a rank one isometry.
Now we apply Lemma \ref{L:not-tilda} and \ref{L:b-f} to $g$ and $W_{\Delta}'$.

If $g$ is in a join subgroup, we project $g$ together with $W_{\Delta}'$
on the infinite factor. The projection of a commutator subgroup is again
a commutator subgroup, thus it is torsion-free. Hence the assumption of
Lemmas \ref{L:not-tilda} and \ref{L:b-f} are satisfied. Thus we apply the same
argument as in Theorem \ref{T:raag} constructing a quasimorphism which can be
extended to $W_{\Delta}'$.
\end{proof}

Before the proof of Theorem \ref{T:pure} let us recall basic properties
and definitions of braid and pure braid groups.
Denote by $D_n$ an open two dimensional disc with $n$ marked points.
The braid group on $n$ strings, denoted $B_n$, is a group of
isotopy classes of orientation-preserving homeomorphisms
of $D_n$ which permute marked points
(this is the mapping class group of a disc with $n$ punctures).
A class of a homeomorphism which fixes all marked points is called a pure braid.
The group of all pure braids on $n$ strings, denoted
$P_n$, is a finite index normal subgroup of $B_n$.

Let $g>1$. Denote by $\mathcal{MCG}_g^n$ the mapping class group of a closed
hyperbolic surface $\Sigma_g$ with $n$ punctures. In \cite{Bir} J. Birman
showed that $B_n$ naturally embeds into $\mathcal{MCG}_g^n$. More precisely,
let $D$ be an embedded disc in $\Sigma_g$ which contains $n$ punctures. Then a
mapping class group of this $n$ punctured disc $D$ is a subgroup of
$\mathcal{MCG}_g^n$. Let us identify $B_n$ with this subgroup. In the same way
we identify $P_n$ with a subgroup of the pure mapping class group
$\mathcal{PMCG}_g^n$. Note that $\mathcal{PMCG}_g^n$ is a finite index subgroup
of $\mathcal{MCG}_g^n$.

It follows from the Nielsen-Thurston decomposition in $\mathcal{MCG}_g^n$ that
for every $\gamma \in B_n<\mathcal{MCG}_g^n$ there exists $N$, pseudo-Anosov
braids $\gamma_1, \gamma_2, \ldots, \gamma_m\in B_n$ and Dehn twists $\delta_1,
\delta_2, \ldots, \delta_n\in B_n$ such that $$\gamma^N = \gamma_1 \gamma_2
\ldots \gamma_m \delta_1^{m_1} \delta_2^{m_2} \ldots \delta_n^{m_n},$$ where
all elements in the above decomposition pairwise commute. Moreover, the support
of each element is contained in a connected component of the disc $D$, is
bounded by a simple curve and contains non empty subset of marked points.

Following \cite[Section 4]{BBF} we call an element $\gamma$ \textbf{chiral} if
it is not conjugate to its inverse. Note that if two elements in
$B_n<\mathcal{MCG}_g^n$ are conjugate in $\mathcal{MCG}_g^n$, then they are
conjugate in $B_n$. Similarly, if two elements in $P_n<\mathcal{PMCG}_g^n$ are
conjugate in $\mathcal{PMCG}_g^n$, then they are conjugate in $P_n$. It follows
that $\gamma$ is chiral in $B_n$ if and only if it is chiral in
$\mathcal{MCG}_g^n$, and the same statement holds for groups $P_n$ and
$\mathcal{PMCG}_g^n$.  The following lemma is a straightforward consequence of
Theorem 4.2 from \cite{BBF}.

\begin{lemma}[Bestvina-Bromberg-Fujiwara]\label{L:fuji-mcg} Let $\Sigma$\/ be a
closed orientable surface, possibly with punctures. Let $G$\/ be a finite index
subgroup of the mapping class group of\/ $\Sigma$. Consider a nontrivial
element\/ $\gamma \in G$ and Nielsen-Thurston decomposition of its appropriate
power as above. Assume that every element from the decomposition is chiral and
nontrivial powers of any two elements from the decomposition are not conjugate
in $G$. Then there is a homogeneous quasimorphism on $G$ which takes a non zero
value on $\gamma$.  \end{lemma}

A group $G$ is said to be \textbf{biorderable} if there exists a linear order
on $G$ which is invariant under left and right translations. For example the
pure braid group on any number of strings is biorderable \cite{RZ}.

\begin{lemma}\label{L:biorder}
Let\/ $G$ be a biorderable group. Then $xy^mx^{-1} \neq y^{-n}$ for every
$y\neq e,x \in G$ and positive $m,n$.
\end{lemma}

In particular, every nontrivial element in a biorderable group is chiral.

\begin{proof}
Let $<$ be a biinvariant order on $G$. Assume on the contrary that
$xy^mx^{-1} = y^{-n}$. Without loss of generality we can assume that $y > e$.
Then $y^m > e$, we can conjugate the inequality by $x$ which gives that $y^{-n}
= xy^mx^{-1} > e$.  Thus $e > y^n$, that is $e > y$. We got a contradiction.
\end{proof}

\begin{proof} [Proof of Theorem \ref{T:pure}]
Let\/ $\gamma$ be a nontrivial pure
braid on $n$ strings. We will show that there is a homogeneous quasimorphism on
$P_n$ nontrivial on $\gamma$. After passing to a power of\/ $\gamma$ we can write
that
$$
\gamma =
\gamma_1 \gamma_2 \ldots \gamma_m \delta_1^{m_1} \delta_2^{m_2} \ldots \delta_n^{m_n}
$$
where $\gamma_i$ and\/ $\delta_i$ are as in the discussion above. Since $P_n$ is
a finite index subgroup of\/ $B_n$ we can find $M$ such that all\/ $\gamma_i^M$
and\/ $\delta_i^M$ are in $P_n$.  Thus passing to even bigger power of\/ $\gamma$
we can assume that the braids arising in the decomposition are pure.

Lemma \ref{L:biorder} implies that every element from the decomposition is
chiral, and so it is chiral in $\mathcal{PMCG}_g^n$.  Let $x$ and $y$ be two
distinct elements among $\gamma_i$ and $\delta_i^{m_i}$.  From the definition
of the decomposition, simple curves associated to $x$ and $y$ bound disjoint
subsets of marked points.  Since isotopy classes from $P_n$ preserve marked
points pointwise, powers of $x$ and $y$ cannot be conjugate by a pure braid and
hence by no element of $\mathcal{PMCG}_g^n$.

The assumptions of Lemma \ref{L:fuji-mcg} are satisfied, hence $\gamma$ is
detectable by homogeneous quasimorphism. Note that this homogeneous
quasimorphism is defined on the whole group $\mathcal{PMCG}_g^n$, and a
quasimorphism on $P_n$, which detects $\gamma$, is a restriction of the above
quasimorphism to the subgroup $P_n<\mathcal{PMCG}_g^n$.
\end{proof}

\section{The bq-dichotomy}\label{S:dichotomy}

The purpose of this section is to prove
the bq-dichotomy for various classes of groups.
We introduce a family of auxiliary groups which detects bounded elements.

\begin{lemma}\label{L:gen-kedra}
Let $\bar{m} = (m_0, m_1, \ldots, m_k)$ be a sequence of integers
such that $\frac{1}{m_0} + \frac{1}{m_1} + \ldots + \frac{1}{m_k} = 0$. Define
$$
\Gamma(\bar{m}) = \langle x_0, \ldots, x_{k}, t_1, \ldots, t_k ~|~
({}^{t_i}\!x_0)^{m_0} = x_i^{m_i}, [x_j,x_k] = e \rangle
$$
Then $g = x_0x_1\ldots x_k$ generates a bounded cyclic subgroup.
\end{lemma}

\begin{proof}
Let $N = m_0m_1\ldots m_k$ and $a_i = \frac{N}{m_i}$.
From the assumption on $m_i$ we have that $a_0+a_1+\ldots+a_k = 0$.
For any $n$ we obtain that
\begin{align*}
g^{nN} & =  x_0^{nN}x_1^{nN}\ldots x_k^{nN}\\
& = x_0^{nm_0a_0}x_1^{nm_1a_1}\ldots x_k^{nm_ka_k}\\
& = x_0^{nm_0a_0} ({}^{t_1}\!x_0)^{nm_0a_1}\ldots ({}^{t_k}\!x_0)^{nm_0a_k}\\
& = x_0^{nm_0(-a_1-a_2-\ldots -a_k)} ({}^{t_1}\!x_0)^{nm_0a_1}\ldots ({}^{t_k}\!x_0)^{nm_0a_k}\\
& = ({}^{t_1}\!x_0)^{nm_0a_1} x_0^{-nm_0a_1}\ldots ({}^{t_k}\!x_0)^{nm_0a_k} x_0^{-nm_0a_k}\\
& = [t_1,x_0^{nm_0a_1}]\ldots [t_k,x_0^{nm_0a_k}].
\end{align*}
It shows that $g^{nN}$ is bounded by $2k$ for every $n$. Hence $g$ generates a bounded subgroup.
\end{proof}

Let us remark that $\Gamma(1,-1)$ is isomorphic to the group $\Gamma$
defined in Section \ref{L:bounded}. We start with a somewhat
weaker statement for Coxeter groups.

\begin{theorem}\label{T:coxeter}
Let $W$ be a Coxeter group and let $g\in W$.
\begin{itemize}
\item
The cyclic subgroup $\langle g \rangle $ is either
bounded or undistorted.
\item
Let $W_T=W_{T_1}\times W_{T_2}$ be a standard parabolic subgroup
such that both $W_{T_1}$ and $W_{T_2}$ are infinite and
standard parabolic.
If the standard projection $W \to W_T$ is well defined
for all $W_T$ of the above form
then $W$ satisfies the bq-dichotomy.
\end{itemize}
\end{theorem}

\begin{remark}
Let $S$ be the standard generating set for $W$.
The property in the second item of the theorem holds if
for every $s\in S\setminus T$ and $t\in T$ the exponent
in the relation $(st)^m$ is even.
\end{remark}

\begin{proof}
We proceed by induction on a number of Coxeter generators. If there is only
one generator the theorem is obvious. Let $n\in \B N$
be a natural number and $W$ be a
Coxeter group generated by $n$ Coxeter generators. Assume that the theorem is true
for Coxeter groups generated by less than $n$ Coxeter generators.
Let $g\in W$ be a nontorsion element. There are two cases:

\textit{Case 1}:
The element $g$  acts as a rank one isometry on the
Davis complex. If no positive power of $g$ is conjugate to a positive power of
$g^{-1}$ then we can apply Lemma \ref{L:b-f} to obtain a homogeneous
quasimorphism nonvanishing on $g$.
Otherwise we have that $xg^mx^{-1} = g^{-n}$ for
some $x \in W$ and positive $m,n\in \B N$.
By Lemma \ref{L:not-tilda} it follows that
$m=n$.  There is a homomorphism
$$\Psi \colon \Gamma(m,-m) \to W$$
defined on generators as
$\Psi(x_0) = \Psi(x_1) = g$, $\Psi(t_1) = x$. Thus the element
$\Psi(x_0x_1) = g^2$ (as well as  $g$)
generates a bounded cyclic subgroup.

\textit{Case 2}:
$g$ does not act as a rank one isometry on the Davis complex.
Then, according to Caprace and Fujiwara \cite[Proposition 4.5]{MR2585575},
$g$ is contained in a parabolic subgroup $P$ that is either
\begin{enumerate}
\item equal to $P_1\times P_2$, where $P_1$ is finite parabolic and
$P_2$ is parabolic and affine of rank at least three or
\item equal to  $P_1\times P_2$, where both $P_1$ and $P_2$ are infinite
parabolic.
\end{enumerate}

In the first case, both $P_1$ and $P_2$ are bounded
\cite{MR2842917} and so is their product and hence $g$
generates a bounded subgroup.

In the second case we project $g$ to the factors and the first
statement follows by induction because the inclusion of a
parabolic subgroup is an isometry due to Corollary \ref{C:isometric}.

If the projection of $g$ to one of the factors is detectable
by a homogeneous quasimorphism then this quasimorphisms extends to the                               product $P$. Thus if W satisfies the assumption of the second statement,
we pull back the latter to W using the projection $W \to W_T$ and
the conjugation $xW_Tx^{-1} = P$. Otherwise $g$ generates a bounded
subgroup in both $P_1$ and $P_2$. Indeed, we proceed by induction since
the assumption of the second statement is inherited by parabolic subgroups.
Thus $g$ also generates a bounded subgroup in $P_1\times P_2$
and hence in $W$.
\end{proof}

\begin{theorem}\label{T:mcg}
The bq-dichotomy holds for a finite index subgroup
of the mapping class group of a closed surface possibly
with punctures.
\end{theorem}

\begin{proof}
Let us recall some notions from \cite[Section 4]{BBF}.  We say that two chiral
elements of a group $G$ are equivalent if some of theirs nontrivial powers are
conjugate.  An equivalence class $\{\gamma_0,\gamma_1, \ldots, \gamma_n\}$ of this relation
is called \textbf{inessential}, if there is a sequence of numbers
$\bar{m} =(m_0,\ldots,m_n)$ such that elements $\gamma_i^{m_i}$ are pairwise conjugate and
$\Sigma \frac{1}{m_i} = 0$.  Let $h = \gamma_0\ldots \gamma_n$, where all
$\gamma_i$'s commute. Note that there is a homomorphism
$$\Psi \colon \Gamma(\bar{m}) \to G$$
defined on the generators as $\Psi(x_i) = \gamma_i$. From the Lemma
\ref{L:gen-kedra} it follows that $\Psi(x_0\ldots x_n) = h$ generates a bounded
subgroup. When $\gamma$ is not chiral,
it generates a bounded subgroup due to a homomorphism from $\Gamma(1,-1)$
defined by $\Psi(x_0)=\Psi(x_1)=\gamma$.

Let $\gamma \in G$. By the same argument as in the proof of Theorem \ref{T:pure} we
can assume that $\gamma$ has a Nielsen-Thurston decomposition within $G$ (that is,
elements of the decomposition are in $G$). Assume that there is no homogeneous
quasimorphism which takes non zero value on $\gamma$.
Then by \cite[Theorem 4.2]{BBF} in the decomposition of $\gamma$ we have either
not chiral elements, or
chiral elements which can be divided into inessential equivalence classes.
Hence we can write that
$$\gamma = c_1\ldots c_m h_1 \ldots h_n,$$
where $c_i$ are not chiral and $h_i$ are products of elements from inessential class. In both cases
they generate bounded subgroups. Since
$c's$ and $h's$ commute, we have that
$$\gamma^k = c_1^k\ldots c_n^k h_1^k \ldots h_n^k.$$
Thus $\gamma$ generates a bounded subgroup in $G$.
\end{proof}

\begin{theorem}\label{T:braid}
The bq-dichotomy holds for Artin braid groups.
\end{theorem}

\begin{proof}
Let
$\gamma \in B_n<\mathcal{MCG}_g^n$, where $g>1$. Recall that two braids in $B_n$ are
conjugate in $B_n$ if and only if they are conjugate in $\mathcal{MCG}_g^n$.
Hence an equivalence class $\{\gamma_0,\gamma_1, \ldots, \gamma_n\}$, where
each $\gamma_i\in B_n$, is essential (respectively inessential) in $B_n$ if and
only if it is essential (respectively inessential) in $\mathcal{MCG}_g^n$.
Similarly if $\gamma$ is not chiral in $B_n$, then it is not chiral in
$\mathcal{MCG}_g^n$.

Assume that there is no homogeneous quasimorphism on $\mathcal{MCG}_g^n$ which
takes non zero value on $\gamma$.  Then by \cite[Theorem 4.2]{BBF} in the
Nielsen-Thurston decomposition of $\gamma$ in $\mathcal{MCG}_g^n$  we have
either
not chiral elements, or chiral elements which can be divided into inessential
equivalence classes. Since each element in the Nielsen-Thurston decomposition
of $\gamma$ lies in $B_n$, and the notion of equivalence class and chirality is
the same in $B_n<\mathcal{MCG}_g^n$ and in $\mathcal{MCG}_g^n$, it follows that
there is no homogeneous quasimorphism on $B_n$ which takes non zero value on
$\gamma$. Hence we can write that
$$
\gamma = c_1\ldots c_m h_1 \ldots h_n,
$$
where $c_i$'s are not
chiral in $B_n$ and $h_i$'s are products of elements from inessential class in
$B_n$. In both cases they generate bounded subgroups (see the discussion in the
proof of the previous case). Since $c's$ and $h's$ commute, we have that
$$\gamma^k = c_1^k\ldots c_n^k h_1^k \ldots h_n^k.$$ Thus $\gamma$ generates a
bounded subgroup in $B_n$.
\end{proof}

\begin{theorem}\label{T:sphere-braid}
The bq-dichotomy holds for spherical braid groups (both
pure and full).
\end{theorem}

\begin{proof}
{\bf The case of spherical pure braid groups $P_n(S^2)$.} Recall that
$P_n(S^2)$ is a fundamental group of an ordered configuration space of $n$
different points in a two-sphere $S^2$. As before we denote by

$\mathcal{MCG}_0^n$ the mapping class group of the $n$ punctured sphere, and by
$\mathcal{PMCG}_0^n$ the pure mapping class group of the $n$ punctured sphere.
Since $P_n(S^2)$ are trivial for $n=1,2$, we assume that $n>2$. It is well
known fact that $P_n(S^2)$ is isomorphic to a direct product of $\B Z/\B Z_2$
and $\mathcal{PMCG}_0^n$, see e.g.  \cite{KV}. Since we already proved the
statement for finite index subgroups of mapping class groups, the proof of this
case follows.

{\bf The case of spherical braid groups $B_n(S^2)$.} The group $B_n(S^2)$ is a
fundamental group of a configuration space of $n$ different points in a
two-sphere $S^2$. It is known that the group $\mathcal{MCG}_0^n$ is isomorphic
to $B_n(S^2)/\langle\Delta^2\rangle$, where $\Delta$ is the Garside fundamental
braid, see \cite{M}. In particular, $\Delta^2$ lies in the center of $B_n(S^2)$
and $\Delta^4=1_{B_n(S^2)}$. Let $$\Pi\colon B_n(S^2)\to
B_n(S^2)/\langle\Delta^2\rangle\cong\mathcal{MCG}_0^n$$ be the projection
homomorphism. Since $\Delta^4=1_{B_n(S^2)}$ and $\Delta^2$ is central, every
homogeneous quasimorphism on $\mathcal{MCG}_0^n$ defines a homogeneous
quasimorphism on $B_n(S^2)$ and vice versa. In addition, if two elements
$\Pi(x),\Pi(y)\in \mathcal{MCG}_0^n$ commute or are conjugate in
$\mathcal{MCG}_0^n$, then $x$ and $y$ commute or are conjugate up to the
multiplication by a torsion element $\Delta^2$ in $B_n(S^2)$.

Let $\gamma\in B_n(S^2)$. Assume that there is no homogeneous quasimorphism on
$B_n(S^2)$ which takes non zero value on $\gamma$. Then there is no homogeneous
quasimorphism on $\mathcal{MCG}_0^n$ which takes non zero value on
$\Pi(\gamma)$. Then by \cite[Theorem 4.2]{BBF} in the Nielsen-Thurston
decomposition of $\Pi(\gamma)$ in $\mathcal{MCG}_0^n$  we have either not
chiral elements, or chiral elements which can be divided into inessential
equivalence classes. Hence we can write that
$$
\Pi(\gamma) = \Pi(c_1)\ldots \Pi(c_m) \Pi(h_1) \ldots \Pi(h_n),$$
where $\Pi(c_i)$ are not chiral in
$\mathcal{MCG}_0^n$ and $\Pi(h_i)$ are products of elements from inessential
class in $\mathcal{MCG}_0^n$. As before, $\Pi(\gamma)$ generates a bounded
subgroup in $\mathcal{MCG}_0^n$ and since $\Delta^4=1_{B_n(S^2)}$ and
$\Delta^2$ is central, $\gamma$ generates a bounded subgroup in $B_n(S^2)$.
\end{proof}

\subsection{The bq-dichotomy for nilpotent groups}
Let us recall that a group $G$ is said to be boundedly
generated if there are cyclic subgroups $C_1,\ldots,C_n$
of $G$ such that $G=C_1\dots C_n$. It is known that
a finitely generated nilpotent group
has bounded generation \cite{MR1444027}.
In the proof below we shall use
a trivial observation that if group is boundedly
generated by bounded cyclic subgroups then it is bounded.

\begin{theorem}\label{T:nilpotent}
Let $N$ be a finitely generated nilpotent group.
Then the commutator subgroup $[N,N]$ is bounded in $N$.
Consequently, $N$ satisfies the bq-di\-cho\-to\-my.
\end{theorem}

In the proof of the theorem we will use the following
observation. Its straightforward proof is left to the reader.
\begin{lemma}\label{L:bb}
Let\/ $K\lhd H< G$ be a sequence of groups such that\/
$K$ is normal in $G$. If\/ $K$
is bounded in $G$ and every cyclic subgroup of\/ $H/K$ is bounded
in $G/K$ then every cyclic subgroup of\/ $H$ is bounded in $G$.
\end{lemma}

\begin{proof}[Proof of Theorem \ref{T:nilpotent}]
Let $N_i\subset N$ be the lower central series. That is
$N_0=N$, $N_1=[N,N]$ and $N_{i+1}=[N,N_i]$. Since $N$
is nilpotent $N_i=0$ for $i>k$ and the last nontrivial
term $N_k$ is central.

Observe first that $N_k$ is bounded in $N$.
Let $x\in N$ and let $y\in N_{k-1}$. Then
$z=[x,y]\in N_k$ is central and a direct calculation shows that
$z^n=[x^n,y]$. Since $N$ is finitely generated,
we know that all $N_i$ are finitely generated as well,
according to  Baer \cite[page 232]{MR0080089}.
Now we have that $N_k$ is finitely generated by
products of commutators of the above form and,
since $N_k$ is abelian, these elements generate
bounded (in $N$) cyclic subgroups. This implies that
$N_k$ is bounded in $N$ as claimed.

The quotient series $N_i/N_k$ is the lower central series
for $N/N_k$ and by the same argument as above we obtain
that $N_{k-1}/N_k$ is bounded in $N/N_k$. Applying
Lemma \ref{L:bb} to the diagram
$$
\xymatrix
{
N_k     \ar[d]\ar[r] & N_k \ar[d]\\
N_{k-1} \ar[d]\ar[r] & N   \ar[d]\\
N_{k-1}/N_k   \ar[r] & N/N_k
}
$$
we get that every cyclic subgroup in
$N_{k-1}$ is bounded in $N$. Again, this
implies that $N_{k-1}$ is bounded in $N$.
Repeating this argument for $N/N_{k-1}$ we obtain
that $N_{k-2}$ is bounded in $N$.
The statement follows by induction.
\end{proof}

\subsection{The bq-dichotomy for solvable groups}

\begin{theorem}\label{T:solvable}
Let\/ $G$ be a finitely generated solvable group
such that its commutator subgroup is finitely
generated and nilpotent.
Then the commutator subgroup $[G,G]$ is bounded in $G$.
Consequently, it satisfies the bq-dichotomy.
\end{theorem}

\begin{proof}
Let us first proof the statement under an additional assumption
that $G$ is metabelian (hence $[G,G]$ is a finitely
generated abelian group).
Let $x,y,t\in G$ and consider the element
$[t,[x,y]]\in [G,[G,G]]$. Observe that
it generates a bounded subgroup in $G$ because
$[x,y]$ commutes with ${}^t[x,y]$ and we
can apply Lemma \ref{L:bounded}.
Since the subgroup $[G,[G,G]]\subset [G,G]$ is finitely generated
abelian, it is boundedly generated by cyclic subgroups
bounded in $G$ and hence $[G,[G,G]]$ is bounded
in $G$.

Consider the following diagram.
$$
\xymatrix
{
G_2= [G,[G,G]] \ar[d]\ar[r] & [G,[G,G]]\ar[d]\\
G_1= [G,G]     \ar[d]\ar[r] & G\ar[d]\\
G_1/G_2        \ar[r]       & G/G_2
}
$$
Since $[G/G_2,[G/G_2,G/G_2]]$ is trivial,
$G/G_2$ is metabelian and nilpotent. Hence,
due to Theorem \ref{T:nilpotent}, we get that
$G_1/G_2=[G/G_2,G/G_2]$ is bounded in $G/G_2$.
It then follows from  Lemma~\ref{L:bb} that
$[G,G]$ is bounded in $G$.

Let us prove the statement for a general $G$. Since the commutator
subgroup $G^1=[G,G]$ is finitely generated and nilpotent
we have, according to Theorem \ref{T:nilpotent}, that
$G^2=[G^1,G^1]$ is bounded in $G_1$ and hence in $G$.
$$
\xymatrix
{
G^2= [G^1,G^1] \ar[d]\ar[r] & [G^1,G^1]\ar[d]\\
G^1= [G,G]     \ar[d]\ar[r] & G\ar[d]\\
G^1/G^2        \ar[r]       & G/G^2
}
$$
Since $G/G^2$ is metabelian and
$G^1/G^2=[G/G^2,G/G^2]$
is finitely generated (because $G^1$ is)
we have that $G^1/G^2$ is bounded in $G/G^2$, due
to the first part of the proof. Again,
by Lemma \ref{L:bb}, we get that $[G,G]$ is bounded
in $G$ as claimed.
\end{proof}

\begin{remark}
The integer lamplighter group $G=\B Z\wr \B Z$ is solvable
and finitely generated but its commutator subgroup is
abelian of infinite rank. The proof still works in this
case because we have that $[G,[G,G]]=[G,G]$. We also showed
it directly in Example \ref{E:bounded} (3).
\end{remark}

\subsection{The bq-dichotomy for graph of groups}
For an introduction to graph of groups see
Serre \cite{MR1954121}.

\begin{lemma} \label{L:g-not-conjugate}
Let\/ $\B{A}$ be a graph of groups and let\/ $G_{\B{A}}$ be its
fundamental group.  Assume tha\/t $g \in G_{\B{A}}$ is not conjugate to an
element of the vertex group.  Then $g$ is either detectable by a homogenous
quasimorphism or $\langle g \rangle$ is bounded with respect to
the conjugation invariant norm.
\end{lemma}

\begin{proof}
Consider the action of $G_{\mathbf{A}}$ on the Bass-Serre tree $T_{\mathbf{A}}$.
The action of $g$ on $T_{\mathbf{A}}$ does not have a fixpoint, for
$G_{\mathbf{A}}$ acts on $T_{\mathbf{A}}$ without edge inversions and
the stabilizers of vertices are conjugate to the vertex groups.  Thus $g$ acts by a
hyperbolic isometry and it is automatically of rank-one. By Lemma \ref{L:b-f}, $g$
is either detectable by a homogeneous quasimorphism, or it is conjugate to
$g^{-1}$, hence $\langle g \rangle$ is bounded.
\end{proof}

Let $H\subset G$ be a subgroup. We say that $H$ satisfies
the {\em relative bq-dichotomy (with respect to $G$)} if every
cyclic subgroup of $H$ is either bounded in $G$ or
it is detected by a homogeneous quasimorphism $q\colon G\to \B R$.
The following result is a
straightforward application of the above lemma.

\begin{theorem} \label{T:graph-of-gps}
Let\/ $\mathbf{A}$ be a graph of groups and let\/ $G_{\mathbf{A}}$ be its fundamental group.
If each vertex subgroup of\/ $G_{\B A}$ satisfies the relative
bq-dichotomy then $G_{\B A}$ satisfies the bq-dichotomy.\qed
\end{theorem}

\begin{example}\label{E:bs}
Baumslag-Solitar groups satisfy the bq-dichotomy.
Indeed, Baumslag-Solitar groups are HNN extensions of the infinite cyclic group $\B Z$.
The graph of groups in this case has one vertex and one edge.
By virtue of Example $\ref{E:bounded}$ (1) the vertex group is bounded and
we can apply Theorem \ref{T:graph-of-gps}.
\end{example}

\begin{example}
The groups $\Gamma(\overline{m})$ (defined in Lemma \ref{L:gen-kedra})
satisfy the bq-dichotomy.
We keep the notation from Lemma \ref{L:gen-kedra}.
The group $\Gamma(\bar{m})$ is the fundamental group of the graph of
groups associated with a rose with $k$ petels. The vertex group is $\B Z^{k+1}$
generated by $x_0,\ldots,x_k$ and the edge groups are cyclic generated
by $x_i^{m_i}$. The elements $x_i$ are detected by a homomorphism
$h \colon \Gamma(\bar{m}) \to \B Z$ defined by $h(t_i) = 0$ and
$h(x_i) = a_i$.
The kernel of this homomorphism is bounded, according to
Lemma \ref{L:gen-kedra}. Consequently,
the bq-dichotomy follows from Theorem \ref{T:graph-of-gps}.
\end{example}

\subsection*{Acknowledgments}
The first author wishes to express his gratitude to Max Planck Institute for
Mathematics in Bonn for the support and excellent working conditions.
He was supported by the Max Planck Institute research grant.
We thank \'Etienne Ghys, \L ukasz Grabowski, Karol Konaszy\'nski and Peter Kropholler
for helpful discussions and the referee for comments and spotting an incorrect
example.

University of Aberdeen supported the visit of MB, SG and MM in Aberdeen
during which a part of the paper was developed.

\section{Erratum: Proposition \ref{P:balanced}}\label{S:erratum}

\hfill{\bf March 2023}

The proof of Proposition \ref{P:balanced} is incomplete.
Namely, the equality
$$
|xr^{-1}ry|_{\times} = |xy|_{\times}
$$
is not justified. Here, we present the proof of the following special case.

\noindent
{\bf Proposition 2.E'.}
{\it
Let $G=\langle S\ |\ R\rangle$
be either a right-angled Artin group or a Coxeter group.
Let $v$ and $w$ be two words in the alphabet $S$ representing
the same element $g\in G$. Then $|v|_{\times}=|w|_{\times}$.
In particular, the cancellation norm of $g$ is equal to the
cancellation length of any word representing $g$.
}

\begin{proof}
Let $G=\langle S\ |\ R\rangle$ be a Coxeter group and let $g\in G$.
Dyer \cite{MR1838781} proved that $\|g\| = |v|_{\times}$, where
$v$ is a shortest reduced word representing $g$. Let $w$ be any word
representing $g$. It follows from \cite[Theorem 3.4.2]{MR2360474}
that $v$ can be obtained from $w$ by a sequence of the following
operations:
\begin{itemize}
\item deleting a subword of the form $ss$;
\item replacing alternating subword $sts\cdots$ of length $m_{st}$
by the alternating word $tst\cdots$ of the same length.
\end{itemize}
The first operation does not decrease the cancellation length. That
is, $|xssy|_{\times} \leq |xy|_{\times}$ for any words $x,y$. The
second operation does not change the cancellation length since
the relation $sts\cdots = tst\cdots$ is balanced. Thus in the process
of reduction from $w$ to $v$ the cancellation length cannot
decrease. Since $|v|_{\times}$ is minimal among all words representing
$g$, we get that $|w|_{\times}=|v|_{\times}$. This shows the statement
for Coxeter groups.

Let $G=\langle S\ |\ R\rangle$ be a right-angled Artin group.
Let $g\in G$ and let $v$ and $w$ be two words representing $g$.
They are related by a sequence of the following operations
\begin{itemize}
\item deleting or inserting a subword of the form $ss^{-1}$;
\item replacing a subword $st$ by the subword $ts$.
\end{itemize}
Since the relation $st=ts$ is balanced, the second operation does not change
the cancellation length. As before, we have that $|xss^{-1}y|_{\times} \leq
|xy|_{\times}$ for any words $x,y$ and any letter $s\in S$.  If a minimal
cancellation sequence of $xss^{-1}y$ does not involve the displayed $s$ or
$s^{-1}$ then $|xss^{-1}y|_{\times}=|xy|_{\times}$.  Suppose that a minimal
cancellation sequence for $xss^{-1}y$ contains $s$ (it can't contain both $s$
and $s^{-1}$ due to its minimality). After deleting the cancelling sequence
we obtain a word $x's^{-1}y'$ representing the trivial element.
Since the process of reduction of $x's^{-1}y'$ involves only replacing
$ab$ with $ba$ or deleting $aa^{-1}$ \cite{MR3587227,MR2322545}, 
there is a letter $s$ either in $x'$
or in $y'$ with which the displayed $s^{-1}$ will cancel. Swapping the
initial $s$ in the cancellation sequence for the $s$ in either $x'$ or $y'$
yield another minimal cancellation sequence that does not involve the
initial $ss^{-1}$. This shows that $|v|_{\times}=|w|_{\times}$ and
finishes the proof.
\end{proof}

\bibliography{bibliography}
\bibliographystyle{acm}

\end{document}